\documentclass[10pt]{amsart}
\usepackage{epsfig}
\usepackage{graphics,graphicx}
\usepackage{amsmath, amsthm, amssymb}

\newtheorem{corollary}{Corollary}[section]

\newtheorem{lemma}[corollary]{Lemma}
\newtheorem{proposition}[corollary]{Proposition}
\newtheorem{remark}[corollary]{Remark}
\newtheorem{theorem}[corollary]{Theorem}
\newcommand{\mylabel}[1]{\label{#1}
            \ifx\undefined\stillediting
            \else \fbox{$#1$}\fi }
\newcommand{\BE}{\begin{equation}}
\newcommand{\BEQ}[1]{\BE\mylabel{#1}}
\newcommand{\EEQ}{\end{equation}}
\newcommand{\rfb}[1]{\mbox{\rm
   (\ref{#1})}\ifx\undefined\stillediting\else:\fbox{$#1$}\fi}

\def\CC{\rm \hbox{C\kern-.56em\raise.4ex
         \hbox{$\scriptscriptstyle |$}\kern+0.5 em }}

\def\be{\begin{equation}}
\def\ee{\end{equation}}

\def\ds{\displaystyle}

\newcommand{\norm}[2]{\|#1 \| _{#2} }

\newfont{\Blackboard}{msbm10 scaled 1200}
\newcommand{\bl}[1]{\mbox{\Blackboard #1}}

\newcommand{\nline}  {{\bl N}}
\newcommand{\rline}  {{\bl R}}

\newcommand{\half}   {{\frac{1}{2}}}

\def\s{\sigma}

\begin{document}

\title[BBM equation on star-shaped networks]{Well-posedness and stabilization of the Benjamin-Bona-Mahony equation on star-shaped networks}

\author{Ka\"{i}s Ammari}
\address{UR Analysis and Control of PDEs, UR 13ES64, Department of Mathematics, Faculty of Sciences of Monastir, University of Monastir, Tunisia and LMV/UVSQ, Universit\'e Paris-Saclay, France}
\email{kais.ammari@fsm.rnu.tn}

\author{Emmanuelle Cr\'epeau}
\address{LMV, UVSQ, CNRS,  Universit\'e Paris-Saclay, 78035 Versailles, France}
\email{emmanuelle.crepeau@uvsq.fr}

\begin{abstract}
We study the stabilization issue of the Benjamin-Bona-Mahony (BBM) equation on a finite star-shaped network with a damping term acting on the central node. In a first time, we prove the well-posedness of this system. Then thanks to the  frequency domain method, we get the asymptotic stabilization result.
\end{abstract}


 \maketitle

\tableofcontents

 
\section{Introduction} \label{secintro}
\setcounter{equation}{0}

The Benjamin-Bona-Mahony equation (BBM), $$u_t-u_{xxt}+u_x+uu_x=0,$$ is a well known model for one-dimensional, small amplitude long waves in nonlinear dispersive systems and an alternative to the Korteweg-de Vries equation, $$u_t+u_{xxx}+u_x+uu_x=0.$$ 
 
As the dispersive term in BBM  "$-u_{xxt}$" produces a strong smoothing effect for time regularity, the existence of solution is easier to prove than for Korteweg-de Vries equation but the controllability properties are not so good. In \cite{micu}, the exact controllability problem of the linearized BBM equation in a bounded domain is studied. Due to the existence of a limit point in the spectrum of the adjoint problem, the exact controllability fails. Recently, Rosier in \cite{rosier} introduced a boundary feedback law to dissipate the energy for BBM on a bounded domain. He proved that if the Unique Continuation Property holds for BBM, then the origin is asymptotically stable for the damped BBM equation.

We consider in this article the Benjamin-Bona-Mahony equation posed on a star-shaped network, with branches of finite length and  a boundary damping term acting on the central node of the network. Our goal is to prove the stabilization of this model depending if the Unique Continuation Property holds. As we study  a finite star-shaped network, we need some conditions at all external edges.  We choose to impose homogeneous Dirichlet conditions. The system to be studied is modelized by the following one where $N$ is the number of edges and on each edge the BBM equation evolves with some transmission conditions at the central node, the continuity condition and a condition on the time derivative of the flux :

\medskip

\begin{equation*}
\leqno(BBM) 
\left \{
\begin{array}{l}
(\partial_t u_{j}- \partial_x^2 \partial_t u_{j} + \partial_x u_j+u_j\partial _x u_j)(t,x)=0,\\
\hfill{~} \forall \, x \in(0,\ell_j),\, t>0,\, j = 1,...,N, \\
u_{j}(t,0)=u_{k}(t,0), \hfill{~} \forall \, j,k = 1,...,N, \\
 \ds \sum_{j=1}^N \partial_{x} \partial_t u_j (t,0) = \alpha \, u_1(t,0)+\frac{N}{3}u_1^2(t,0), \hfill{~} \forall \, t > 0, \\
u_j(t,\ell_j) = 0, \hfill{~} \forall \, t > 0, \, j= 1,...,N, \\
u_j(0,x)=u_j^0(x), \hfill{~} \forall \, x \in (0,\ell_j),\,   j=1,...,N,
\end{array}
\right.
\end{equation*}
where $\alpha$ is a constant satisfying 
\begin{equation}\label{alpha}\alpha > \frac{N}{2}.\end{equation}

\medskip

In the last few years various physical models of multi-link flexible structures consisting of finitely many interconnected flexible elements such as strings, beams, plates, shells have been mathematically studied.  For details about physical motivation for the models, see \cite{dagerzuazua} and the references therein. 
For interconnected strings and beams, a lot of results have been obtained (see for example \cite{ammari4}, \cite{amjel}, \cite{dagerzuazua}). Recently, we have proved the exponential stabilization of the Korteweg-de Vries equation on a star-shaped network, see \cite{ammari-crepeau}, with a boundary damping term on the central node. In that case, the stabilization occurs when there is at most one critical length in the network.

\medskip
\section{Notations and presentation of the problem}\label{finite}

Now, let us first introduce some notations and definitions which will be used throughout the rest of the paper, in particular some are linked to the notion of $C^{\nu }$- networks, $\nu \in \nline$ (as introduced in \cite{dagerzuazua}). 

\medskip

Let $\Gamma$ be a connected topological graph embedded in $\rline$, with $N$ edges ($N \in \nline^{*}$).  
Let $K=\{k_{j}\, :\, 1\leq j\leq N\}$ be the set of the edges of $\Gamma$. Each edge $k_{j}$ is a Jordan curve in $\rline$ and is assumed to be parametrized by its arc length $x_{j}$ such that
the parametrization $\pi _{j}\, :\, [0,\ell_j]\rightarrow k_{j}\, :\, x_{j}\mapsto \pi _{j}(x_{j})$ is $\nu$-times differentiable, i.e. $\pi _{j}\in C^{\nu }([0,\ell_j],\rline)$ for all $1\leq j\leq N$. 
The $C^{\nu}$- network $\mathcal{T}$  associated with $\Gamma$ is then defined as the union $${\mathcal T}=\bigcup _{j=1}^{N}k_{j}.$$

We study here the stabilization problem of a Benjamin-Bona-Mahomy (see \cite{BBM, BT,rosier}) system on a star-shaped network as in the following figure \ref{fig} for $N =3$. 
\begin{figure}[htbp]
\begin{center} \label{fig}
\includegraphics[scale=1.40]{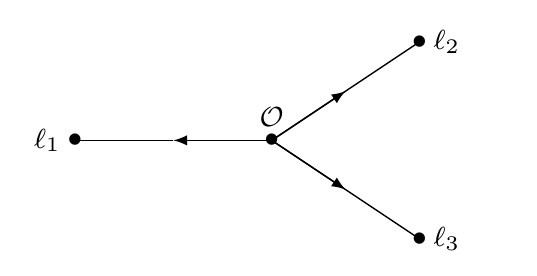}
\caption{Star-Shaped Network for $N =3$}
\end{center}
\end{figure}

More precisely, for each edge $k_{j}$, the scalar function $u_j(t,x)$ for $x \in (0,\ell_j)$ and $t > 0$
contains the information on the displacement of the wave at location $x$ and time $t$, $1 \leq j \leq N$. 

\smallskip

Our problem is the following one : 

\noindent For initial condition $\underline u_0$ in a space to be determined, can we prove the existence of a solution of (BBM) $\underline u$ for all $t>0$ such that $\underline u$ tends to $0$ when $t$ tends to infinity in a space to be determined ?

\begin{remark}
We can get the same stabilization result with (BBM) posed on a tree network, but for a sake of clarity, we only prove our result on a star-shaped network. In the case of a tree, the system can   modelized  the human cardiovascular system for example.
\end{remark}

\medskip

Next section concerns the well-posedness of the solutions of the linearized version of $(BBM)$ called $(LBBM)$.  Then we prove the strong stability of the energy $E(t)$ of the solutions of this linearized system . Then we study the nonlinear version.

Our technique is based on a frequency domain method.
\section{Well-posedness of  systems $(LBBM)$ and $(BBM)$ } \label{wellposed}

In order to study system $(BBM)$ we need a proper functional setting. 
We define the following spaces:  
\begin{itemize}
\item On each edge : 
$$H^1_r (0,\ell_j) = \left\{v \in H^1(0,\ell_j), \, v(\ell_j) = 0 \right\}.$$ 
\item On the whole network
 \begin{itemize}
\item ${\mathbb L}^2(\mathcal{T}) = \ds \oplus_{j=1}^{N} L^2(0,\ell_j), 
$
\item ${\mathbb H}^1_r(\mathcal{T}) = \ds \oplus_{j=1}^{N} H^1_r(0,\ell_j), $

\item $\mbox{for } 1/2<s \leq 3/2$,  
\begin{multline*}{\mathbb H}^s_e (\mathcal{T}) = \bigg\{\underline{u}=(u_1,...,u_N) \in \ds \oplus_{j=1}^N H^s(0,\ell_j), \\
u_j(l_j)=0,\, u_j(0) = u_k (0), \, \forall \, j,k =1, \ldots ,N \bigg\},
\end{multline*}
where ${\mathbb H}^1_e (\mathcal{T}) $ is equipped with the inner product, $ \forall \, \underline{u}, \underline{v} \in {\mathbb H}^1_e (\mathcal{T})$
\begin{equation}\label{ipV}
<\underline{u},\underline{v}>_{{\mathbb H}^1_e (\mathcal{T})} = 
\ds \sum_{j=1}^{N} \int_0^{\ell_j} \left( u_j \overline{v}_j + \frac{du_{j}}{dx}(x)  \overline{\frac{dv_j}{dx}}(x)  \right) \, 
dx.
\end{equation}
\item $\mbox{for } 3/2<s$,  
\begin{multline*}{\mathbb H}^s_e (\mathcal{T}) = \bigg\{\underline{u}=(u_1,...,u_N) \in \ds \oplus_{j=1}^N H^s(0,\ell_j),  u_j(\ell_j)=0,\\
u_j(0) = u_k (0), \sum_{j=1}^N \partial_x u_j(0)=0 \, \forall \, j,k =1, \ldots ,N \bigg\}.
\end{multline*}

\end{itemize}\end{itemize}
In order to prove the well posedness of $(BBM)$, we follow the work of Rosier \cite{rosier}. Let $\underline u^0\in \mathbb H^1_e(\mathcal T)$ and $\underline v=\underline u_t$. Then $\underline v$ solves the following elliptic problem, 

\begin{equation}\label{elliptic}
\left \{
\begin{array}{l}
\ds(1-\partial_x^2 )v_{j}= - \partial_x u_j-u_j\partial _x u_j:=f_j,\\
\hfill{~}\forall \, x \in(0,\ell_j),\, j = 1,\ldots,N, \\
v_{j}(0)=v_{k}(0),\hfill{~} \forall \, j,k = 1,\ldots,N, \\
 \ds \sum_{j=1}^N \partial_{x}v_j (0) = \alpha \, u_1(0)+\frac{N}{3}u_1^2(0):=a, \\
v_j(\ell_j) = 0, \hfill{~} \forall \, j= 1,\ldots,N.
\end{array}
\right.
\end{equation}

Let us define for all $j=1,\ldots,N$, the function $\varphi_j$ by $\varphi_j(x)=\displaystyle\frac{x-\ell_j}{\ell_j} \left(\sum_{i=1}^N \ell_i^{-1} \right)^{-1}$. Thus $\underline\varphi=(\varphi_1,\ldots,\varphi_N)$ satisfies,
\begin{equation*}
\left\{\begin{array}{l}
\varphi_j(\ell_j)=0,\hfill{~}\forall j=1,\ldots,N,\\
\ds\varphi_j(0)= - \left(\sum_{i=1}^N \ell_i^{-1} \right)^{-1} = \varphi_k(0),\hfill{~}\forall j,k=1,\ldots,N,\\
\ds\sum_{j=1}^N \partial_x\varphi_j(0)=1.
\end{array}\right.
\end{equation*}

Then we can write the solution of the elliptic problem \eqref{elliptic} as $\underline v=\underline w+a\underline \varphi$, where $\underline w$ solves the homogeneous elliptic problem :

\begin{equation}\label{elliptic2}
\left \{
\begin{array}{l}
\ds(1-\partial_x^2 )w_{j}=f_j-(1-\partial_x^2 )\varphi_{j}=q_j,\\
\ds \hfill{~}\forall \, x \in(0,\ell_j),\, j = 1,\ldots,N, \\
w_{j}(0)=w_{k}(0),\hfill{~} \forall \, j,k = 1,\ldots,N, \\
 \ds \sum_{j=1}^N \partial_{x}w_j (0) =0, \\
w_j(\ell_j) = 0, \hfill{~} \forall \, j= 1,\ldots,N.
\end{array}
\right.
\end{equation}
Let us define the operator $$
\Delta_{\mathcal{T}} \underline{u}:= 
\left( 
\begin{array}{c}
\frac{d^2 u_1}{dx^2}  \\
 ...
\\
\frac{d^2 u_N}{dx^2}
\end{array}
\right), \,
\forall \, \underline{u} \in \mathcal{D}(\Delta_{\mathcal{T}}) := \mathbb{H}^2_e (\mathcal{T}).
$$

We can easily  prove by usual variational theory and Lax-Milgram theorem that we may define the operator $$(I-\Delta)^{-1}_\mathcal T:\mathbb L^2(\mathcal T)\rightarrow \mathbb H^1_r(\mathcal T)\cap \mathbb H^2_e(\mathcal T)$$ such that $(I-\Delta)^{-1}_\mathcal T(\underline q)=\underline w$ is solution of \eqref{elliptic2}.

Thus the solution $\underline u$ of $(BBM)$ can be written in its integral form as 
\begin{multline}
\underline u(t)=\Gamma(\underline u)(t):=\underline u^0+\int_0^t \left[-(I-\Delta)^{-1}_\mathcal T(\partial_x\underline u+\underline u\partial_x \underline u)(\tau)\right.\\
\left.+\bigg(I-(I-\Delta)^{-1}_\mathcal T(I-\Delta_\mathcal T)\bigg)\left(\varphi_j\right)_j(\alpha \, u_1+\frac{N}{3}u_1^2)(\tau,0)\right]d\tau.
\end{multline}

Let $R>0$ to be chosen latter and $B_R$ the closed ball in $C([0,T],  \mathbb H^1_r(\mathcal T))$ of center 0 and radius $R$.
As $$(I-\Delta)^{-1}_\mathcal T\partial_x:\mathbb H^1_r(\mathcal T)\rightarrow \mathbb H^1_r(\mathcal T),$$ we can easily prove that for $\underline u,\underline v\in B_R$ we have the estimates,
\begin{equation*}
\begin{split}
\ds&\norm{\Gamma(\underline u)(t)-\Gamma(\underline v)(t)}{\mathbb H^1_r(\mathcal T)}\leq CT(1+R)\norm{u-v}{C([0,T], \mathbb  H^1_r(\mathcal T))},\\
\ds&\norm{\Gamma(\underline u)(t)}{\mathbb H^1_r(\mathcal T)}\leq \norm{\underline u^0}{\mathbb H^1_r(\mathcal T)}+CT(1+R)\norm{u}{C([0,T],  \mathbb H^1_r(\mathcal T))}.
\end{split}
\end{equation*}

By chosing $R=2\norm{\underline u^0}{\mathbb H^1_r(\mathcal T)}$ and $T=(2C(1+R))^{-1}$, we can apply the Banach fixed point theorem and we have the existence of a unique solution of $(BBM)$. By using the same argument as \cite{rosier}, thanks to the energy inequality we get the existence of a solution for all time $T$.

\begin{proposition}\label{wpBBMF}
Let $\underline u^0\in \mathbb H^1_e(\mathcal T)$, then there exists a unique solution $\underline u\in C([0,\infty), \mathbb H^1_e(\mathcal T))$ of $(BBM)$. \end{proposition}

If we consider the linearized version of $(BBM)$, 

\begin{equation*}
\leqno(LBBM) 
\left \{
\begin{array}{l}
(\partial_t u_{j}- \partial_x^2 \partial_t u_{j} + \partial_x u_j)(t,x)=0,\\
\hfill{~}\forall \, x \in(0,\ell_j),\, t>0,\, j = 1,...,N, \\
u_{j}(t,0)=u_{k}(t,0),\hfill{~} \forall \, j,k = 1,...,N, \\
 \ds \sum_{j=1}^N \partial_{x} \partial_t u_j (t,0) = \alpha \, u_1(t,0), \hfill{~} \forall \, t > 0, \\
u_j(t,\ell_j) = 0, \hfill{~} \forall \, t > 0, \, j= 1,...,N, \\
u_j(0,x)=u_j^0(x), \hfill{~}\forall \, x \in (0,\ell_j),\,   j=1,...,N,
\end{array}
\right.
\end{equation*}
then by using the same proof as before and usual linear theory we get the following existence result for all time $T$.

\begin{proposition}\label{3exist1} 
For an initial datum $\underline u^0\in \mathbb{H}^1_e(\mathcal{T})$, there exists a unique solution $\underline u\in C([0,\,+\infty),\, \mathbb{H}^1_e(\mathcal{T})) $ to problem $(LBBM)$.
Moreover, the solution $\underline{u}$ satisfies \rfb{dissipae1}.
Therefore the energy is decreasing.
\end{proposition}
We then define the operator $\mathcal A\in \mathcal L(\mathcal D(\mathcal A), \mathbb H^1_e(\mathcal T))$ by
\begin{equation*}
\mathcal A \underline u=-(I-\Delta)_\tau^{-1}\partial_x \underline u+\alpha\bigg(I-(I-\Delta)_\tau^{-1}(I-\Delta_\tau)\bigg)\underline \varphi u_1(0), \, \forall \underline u\in \mathcal D(\mathcal A),
\end{equation*}
where $\ds\mathcal D(\mathcal A)=\{\underline u\in \mathbb{H}^1_e(\mathcal{T}), \, \sum_{i=1}^n\partial_x(\mathcal A\underline u)_i(0)=\alpha u_1(0)\}$.
\section{Strong stability of $(LBBM)$} \label{resolvent}
We define the natural energy $E(t)$ of a solution $\underline{u} = (u_1,...,u_N)$ of $(LBBM)$ system by
\be 
\label{energy1}
E(t)=\frac{1}{2} \ds \sum_{j=1}^{N} \left( \int_{0}^{\ell_j} \left(|u_{j}(t,x)|^2+ |\partial_x u_{j}(t,x)|^2\right){\rm d}x \right).
\ee

We can easily check that every sufficiently smooth solution of $(LBBM)$ satisfies the following dissipation law, with our choice of $\alpha$, \eqref{alpha}.
\begin{equation}\label{dissipae1}
E^\prime(t) = - \ds \left(\alpha - \frac{N}{2}\right) \, \ds \bigl|u_{1}(t,0)\bigr|^2 \leq 0, 
\end{equation}
and therefore, the energy is a nonincreasing function of the time variable $t$.
We prove a decay result of the energy of system $(LBBM)$, under condition on $\ell_j, j=1,...,N$, for all initial data in the energy space. Our technique is based on a frequency domain method and precisely
we make use of the following result due to Arendt and Batty \cite{arendtbatty}:

\begin{theorem}\label{thmArendtBatty}
Let $(T(t))_{t\geq 0}$ be a bounded $C_0$-semigroup on a reflexive space $X$. Denote by $A$ the generator of $(T(t))$ and by $\sigma(A)$ the spectrum of $A$. If $\sigma(A)\cap i\mathbb{R}$ is countable and no eigenvalue of $A$ lies on the imaginary axis, then $\ds \lim_{t\rightarrow+\infty} T(t)x = 0$ for
all $x\in X$.
\end{theorem}

The main result of this section is given by the following result: 

\begin{theorem} \label{lr}
The system $(LBBM)$ is strongly stable, i.e., for all 
$\underline{u}^0 \in \mathbb{H}^1_e (\mathcal{T})$ we have 
\BEQ{EXPDECEXP3nb}
  \lim_{t \rightarrow + \infty} \left\| e^{t{\mathcal A}} \underline{u}^0 \right\|_{\mathbb{H}^1_e (\mathcal{T})}  = 0.
\EEQ
if and only if 
\be
\label{condstab}
\frac{\ell_i}{\ell_j} \notin \mathbb{Q}, \, \forall \, 1 \leq i \neq j \leq N,
\ee 
\end{theorem}
  
{\it Proof.}
By Theorem \ref{thmArendtBatty}, the proof of Theorem \ref{lr} is based on the following lemma.

\begin{lemma} \label{condsp}
The discrete spectrum of ${\mathcal A}$ contains no point on the imaginary axis if and only if condition \rfb{condstab} is satisfied. 
\end{lemma}

\begin{proof}
Since ${\mathcal A}$ is compact,  $0 \notin \rho({\mathcal A})$.
Its spectrum $\s({\mathcal A})$ only consists of eigenvalues of ${\mathcal A}$ and $0$, i.e., 
$$
\s({\mathcal A}) = \sigma_{pp} (\mathcal{A}) \cup \sigma_c (\mathcal{A}) = \sigma_{pp} (\mathcal{A}) \cup \left\{ 0 \right\}.
$$
We will
show that the equation
\be 
{\mathcal A} Z = i \beta \, Z
\label{1.10}
\ee
with $Z= \underline{y} \in {\mathcal D}({\mathcal A})$ and $\beta \neq 0$ admits only the trivial solution.

By taking the
inner product of (\ref{1.10}) with $Z \in {\mathbb H}^1_e(\mathcal{T})$ and using
\be
\label{1.7}
\Re <{\mathcal A}Z,Z>_{{\mathbb H}^1_e(\mathcal{T})} = - \, \left(\alpha - \frac{N}{2} \right) \, \left| y_j (0)\right|^2,
\ee
we obtain
that $y_j(0)=0, \, \forall \, j=1,...,N$. Next, we get
to a second order ordinary differential system:
\be
\left\{ \begin{array}{l} 
i \beta y_j - i\beta \, \ds \frac{d^2 y_j}{dx^2} + \ds \frac{dy_j}{dx} = 0, \hfill{~} \forall(0,\ell_j), \, j = 1,...,N,\\
y_j(0) = y_j(\ell_j)=0, \hfill{~} \forall \, j=1,...,N, \\
\ds \sum_{j=1}^N \frac{dy_j}{dx}(0) = 0.
\end{array} 
\right.
\label{1.11}
\ee

\begin{itemize}
\item
If $|\beta| > \half$
\begin{equation*}
y_j(x) = A_j \, e^{- \frac{i}{2\beta} \, x } \, \sinh \left(x \, \sqrt{1 - \frac{1}{4\beta^2}} \right), \, x \in (0,\ell_j),
\end{equation*}
\begin{equation*}
y_j(\ell_j) = 0 \Rightarrow A_j = 0, \forall \, j=1,...,N,
\end{equation*}
then $\underline{y} = 0$.
\item
If $|\beta| = \half$
\begin{equation*}
y_j(x) = 0 \Rightarrow A_j = 0, \forall \, j=1,...,N,
\end{equation*}
then $\underline{y} = 0$.
\item
If $|\beta| < \half$ then
\begin{equation*}
y_j(x) = i \, A_j \, e^{- \frac{i}{2\beta} \, x } \, \sin \left(x \, \sqrt{\frac{1}{4\beta^2} - 1} \right), \, x \in (0,\ell_j),
\end{equation*}
\begin{multline*}
y_j(\ell_j) = 0, \, \ds \sum_{j=1}^N \frac{dy_j}{dx}(0) = 0 \Leftrightarrow\\
A_j \, \sin \left(\ell_j \, \sqrt{\frac{1}{4\beta^2} - 1} \right) = 0, \, \ds \sum_{j=1}^N A_j = 0
\end{multline*}
\begin{equation*}
\Rightarrow \exists \, 1 \leq i \neq j \leq N \; \hbox{such that} \; \ell_i / \ell_j \in \mathbb{Q},
\end{equation*}
and for 
\begin{multline*}
\ell_1/\ell_2 = \frac{p}{q}, \, p,q \in \mathbb{N}^*\\
 \beta = \frac{\ell_1}{2} \, \sqrt{\frac{1}{p^2 \pi^2 + \ell_1^2}} = 
\frac{\ell_2}{2} \, \sqrt{\frac{1}{q^2 \pi^2 + \ell_2^2}}, \end{multline*}
\begin{equation*}
\underline{y} (x) = \left(
\begin{array}{c}
 e^{- \frac{i}{2\beta} x} \, \sin\left(\frac{p \pi}{\ell_1} x \right) \\
- e^{- \frac{i}{2\beta} x} \, \sin\left(\frac{p \pi}{\ell_1} x \right) \\
\bf 0_{N-2,1}
\end{array}
\right).
\end{equation*}
\end{itemize}

So, system \rfb{1.11} admits only the trivial solution if and only if condition \rfb{condstab} is satisfied.

\end{proof}
With Lemma \ref{condsp} we easily deduce the stability result \eqref{EXPDECEXP3nb}.
 
\section{Stabilization of $(BBM)$}

In this section, we study the following nonlinear dissipative $(BBM)$ system:

\begin{equation*}
\leqno(BBM) 
\left \{
\begin{array}{l}
(\partial_t u_{j}- \partial_x^2 \partial_t u_{j} + \partial_x u_j + u_j \partial_x u_j)(t,x)  =0,\\
\hfill{~}\forall \, x \in(0,\ell_j),\, t\in(0,\infty),\, j = 1,...,N, \\
u_{j}(t,0)=u_{k}(t,0),\hfill{~} \forall \, j,k = 1,...,N, \, t > 0, \\
\ds \sum_{j=1}^N \partial_{x} \partial_t u_j (t,0) = \alpha \, u_1(t,0) + \frac{N}{3} \, u_1^2(t,0), \hfill{~} \forall \, t > 0, \\
u_j(t,\ell_j) = 0, \hfill{~} \forall \, t > 0, \, j= 1,...,N, \\
u_j(0,x)=u_j^0(x), \hfill{~} \forall \, x \in (0,\ell_j),\,   j=1,...,N,
\end{array}
\right.
\end{equation*}
where $\alpha > \frac{N}{2}$.

\begin{theorem} \label{BBM}
For all $\underline{u}^0 \in \mathbb{H}^1_e (\mathcal{T})$ there exists a unique solution $\underline{u} \in C([0,+\infty), \mathbb{H}^1_e (\mathcal{T}))$ of $(BBM)$ system and it satisfies 
\begin{equation*}
\label{EI}
E^\prime (t) = - \, \left(\alpha - \frac{N}{2}\right) \left|u_j(t,0)\right|^2, \, j =1,...,N.
\end{equation*}
Moreover 
\begin{equation}
\label{AS}
\underline{u}(t) - \underline{v}(t) \rightharpoonup 0 \; \hbox{in} \; \mathbb{H}^1_e (\mathcal{T}),
\end{equation}
where $\underline{v} \in C([0,+\infty), \mathbb{H}^1_e (\mathcal{T}))$ is a solution of 
$$
\left\{
\begin{array}{ll}
\partial_t v_j - \partial_x^2 \partial_t v_j + \partial_x v_j + v_j \partial_x v_j = 0, \\
\hfill{~} x \in (0,\ell_j), t \in (0,+\infty), \\
v_j(t,0) = v_j(t,\ell_j) = 0, \hfill{~} t \in (0,+\infty), \, j=1,...,N, \\
\ds \sum_{j=1}^N \partial_{x} \partial_t u_j (t,0) = 0, \hfill{~} t \in (0,+\infty).
\end{array}
\right.
$$
We note that if the unique continuation property:

\medskip

$$
\left\{ 
\begin{array}{ll}
\hbox{for any} \;  u^0 \in H^1 (\mathbb{R}), \; \hbox{if the solution } u  \; \hbox{of} \\

\left\{
\begin{array}{ll}
\partial_t u - \partial^2_x \partial_t u + \partial_x u + u \partial_x u = 0, \, (t,x) \in \mathbb{R}^+ \times \mathbb{R}, \\
u(0,x) = u^0(x), \, x \in \mathbb{R} \end{array} \right.
\\
\hbox{satisfies} \; u(t,x) = 0, \, \forall \, (t,x) \in (0,T) \times \omega \;  \\
\hbox{for some non-empty open set} \; \omega \subset \mathbb{R} \\
 \hbox{and some} \; T > 0, \; \hbox{then}  \;  u = 0 
\end{array}
\right.
$$
\medskip

is satisfied, then 

\be
\label{ASb}
\underline{u}(t) \rightharpoonup 0 \; \hbox{in} \; \mathbb{H}^1_e (\mathcal{T}).
\ee

\end{theorem}

\begin{proof}

The  well-posedness as been done in Theorem \ref{wpBBMF}. We now prove \eqref{AS}-\eqref{ASb}

Let $(t_n)_{n \geq 0}$ be a sequence such that $t_n \rightarrow + \infty$ as $n \rightarrow + \infty$. We can assume that $t_{n+1} - t_n \geq T, \, \forall \, n \geq 1$ (by extracting a subsequence if necessary) and that, for some $v^0 \in \mathbb{H}^1_e (\mathcal{T}), \underline{u}(t) \rightharpoonup \underline{v}^0$ in $\mathbb{H}^1_e (\mathcal{T})$ as $n \rightarrow + \infty$. According the continuity of the flow map we have that:
$$
\underline{u}(t_n + \cdot) \rightharpoonup \underline{v}(\cdot) \; \hbox{in} \;  C([0,T],\mathbb{H}^1_e (\mathcal{T})),
$$
where $\underline{v}$ is a solution of $(BBM)$ issued from $\underline{v}^0$ at $t=0$. 
$$
E(t_{n+1}) - E(t_n) + \left(\alpha - \frac{N}{2} \right) \, \int_{t_n}^{t_{n+1}} \left|u_j(t,0)\right|^2 \, dt = 0.
$$
Which implies, $n \rightarrow + \infty$, 
$$
\int_{0}^{T} \left|u_j(t,0)\right|^2 \, dt = 0, \, \forall \, j =1,...,N.
$$
By extending $v_j$ by $0$ for $x \in \mathbb{R} \setminus (0,\ell_j), j=1,...,N,$ and $t \in (0,T)$, 
$$
\left\{
\begin{array}{l}
\partial_t v_j - \partial_x^2 \partial_t v_j + \partial_x v_j + v_j \partial_x v_j = 0, \,x \in \mathbb{R}, t \in (0,T), \\
v_j(t,x) = 0, \hfill{~} x \in \mathbb{R} \setminus (0,\ell_j), t \in (0,T), \, j=1,...,N.
\end{array}
\right.
$$
From the unique continuation property hypothesis, we have $v_j = 0, j=1,...,N.$
\end{proof}

As consequence and according to \cite[Theorem 2.3]{rosier}, we have the following corollary.

\begin{corollary}
Let $s > 1/2$. Then, for all $\underline{u}^0 \in \mathbb{H}^s_e (\mathcal{T})$  there exists a unique solution $\underline{u} \in C([0,+\infty), \mathbb{H}^s_e (\mathcal{T}))$ of $(BBM)$ system.  
Moreover 
\begin{equation*}
\label{ASs}
\underline{u}(t) - \underline{v}(t) \longrightarrow 0 \; \hbox{in} \; \mathbb{H}^s_e (\mathcal{T}), \, \forall \, 1/2 < s < 1,
\end{equation*}
where $\underline{v} \in C([0,+\infty), \mathbb{H}^s_e (\mathcal{T}))$ is a solution of 
$$
\left\{
\begin{array}{l}
\partial_t v_j - \partial_x^2 \partial_t v_j + \partial_x v_j + v_j \partial_x v_j = 0, \\
\hfill{~} x \in (0,\ell_j), t \in (0,+\infty), \\
v_j(t,0) = v_j(t,\ell_j) = 0, \hfill{~} t \in (0,+\infty), \, j=1,...,N, \\
\ds \sum_{j=1}^N \partial_{x} \partial_t v_j (t,0) = 0, \hfill{~} t \in (0,+\infty), \\
v_j(0,x) = v_j^0 (x), \hfill{~} x \in (0,\ell_j), \, j=1,...,N.
\end{array}
\right.
$$
And if the unique continuation property (see Theorem \ref{BBM}) is satisfied, then 
\begin{equation*}
\label{ASbs}
\underline{u}(t) \longrightarrow 0 \; \hbox{in} \; \mathbb{H}^s_e (\mathcal{T}), \, \forall \, 1/2 < s < 1.
\end{equation*}
\end{corollary}

\bibliographystyle{siam}

\end{document}